\documentclass[11pt]{amsart}
\usepackage{geometry}
\usepackage{graphicx}
\usepackage{amssymb}
\usepackage{epstopdf}
\usepackage{amsmath,amscd}
\usepackage{amsthm}
\usepackage{enumerate}
\usepackage{url,verbatim}
\usepackage{subfig}
\usepackage{enumitem}

\RequirePackage[colorlinks,citecolor=blue,urlcolor=blue]{hyperref}
\usepackage{breakurl}
\theoremstyle{plain}
\newtheorem{theorem}{Theorem}
\newtheorem{definition}[theorem]{Definition}
\newtheorem{lemma}[theorem]{Lemma}
\newtheorem{proposition}[theorem]{Proposition}

\newcommand\ol{\overline}
\newcommand\Aut{\mathrm{Aut}}
\newcommand \sA{\mathcal{A}}

\newcommand\RR{{\mathbb R}}

\newcommand\HH{{\mathbb H}}

\renewcommand\ell{l}

\newcommand\bS{\mathbb{S}}

\newcounter{mycount}

\numberwithin{equation}{section}
\numberwithin{theorem}{section}
\numberwithin{figure}{section}

\title{Ising Percolation on Nonamenable Planar Graphs}

\date{}

\author{Zhongyang Li}
\address{Department of Mathematics,
University of Connecticut,
Storrs, Connecticut 06269-3009, USA}
\email{zhongyang.li@uconn.edu}
\urladdr{\url{https://mathzhongyangli.wordpress.com}}

\begin{document}
\maketitle

\begin{abstract}We study infinite ``$+$'' or ``$-$'' clusters for an Ising model on an connected, transitive, non-amenable, planar, one-ended graph $G$ with finite vertex degree. If the critical percolation probability $p_c^{site}$ for the i.i.d.~Bernoulli site percolation on $G$ is less than $\frac{1}{2}$, we find an explicit region for the coupling constant of the Ising model such that there are infinitely many infinite ``$+$''-clusters and infinitely many infinite ``$-$''-clusters, while the random cluster representation of the Ising model has no infinite 1-clusters. If $p_c^{site}>\frac{1}{2}$, we obtain a lower bound for the critical probability in the random cluster representation of the Ising model in terms of $p_c^{site}$.
\end{abstract}

\section{Introduction}

\subsection{Percolation}
A percolation model on a graph $G=(V(G),E(G))$ is a probability measure on the sample space $\Omega$ consisting of all the subsets of vertices of $G$, that is,
\begin{eqnarray*}
\Omega=\{0,1\}^{V(G)}
\end{eqnarray*}
Each subset of vertices is called a percolation configuration. A vertex $v\in V(G)$ has state 1 in the configuration if it is included in the subgraph; otherwise it has state 0 in the configuration.

Let $\omega\in \Omega$. A cluster in $\omega$ is a maximal connected set of vertices in which each vertex has the same state in $\omega$. A cluster is a 0-cluster (resp.\ 1-cluster) if every vertex in it has state 1 (resp.\ state 0).  A cluster is finite (resp.\ infinite) if there are finitely many vertices (resp.\ infinitely many vertices) in the cluster. We say that percolation occurs in $\omega$ if there exists an infinite 1-cluster in $\omega$. 

One of the most studied percolation models is the Bernoulli($p$) site percolation, where $p\in[0,1]$. The Bernoulli($p$) site percolation on $G$ is a random subset of $V(G)$, such that each vertex is in the percolation subset with probability $p$ independently. The critical probability for Bernoulli($p$) site percolation model on $G$ is defined by
\begin{eqnarray*}
p_c^{site}(G):=\inf\{p\in[0,1]: \mathrm{Bernoulli}(p)\ \mathrm{site\ percolation\ on}\ G\ \mathrm{has\ an\ infinite}\ 1-\mathrm{cluster\ a.~s.}\}.
\end{eqnarray*}
See \cite{GP} for background on percolation.

\subsection{Graph Theory}
Let $G$ be a graph with automorphism group $\Aut(G)$. The graph $G$ is called vertex-transitive, or transitive, if there exists a subgroup $\Gamma\subseteq \mathrm{Aut}(G)$, such that all the vertices are in the same orbit under the action of $\Gamma$ on $G$. The graph $G$ is called quasi-transitive if there exists a subgroup $\Gamma\subseteq \mathrm{Aut}(G)$, such that all the vertices are in finitely many different orbits under the action of $\Gamma$ on $G$. 

Assume $\Gamma\subseteq \mathrm{Aut}(G)$ acts on $G$ quasi-transitively. We say the action of $\Gamma$ on $G$ is unimodular if for any $u,v\in V(G)$ in the same orbit of $\Gamma$,
\begin{eqnarray*}
|\mathrm{Stab}_u(v)|=|\mathrm{Stab}_v(u)|.
\end{eqnarray*}
where $\mathrm{Stab}_u$ is the subgroup of $\Gamma$ defined by
\begin{eqnarray*}
\mathrm{Stab}_u:=\{\gamma\in \Gamma: \gamma(u)=u\}.
\end{eqnarray*}

A graph $G$ is called non-amenable if
\begin{eqnarray}
\inf_{K\subseteq V(G), |K|<\infty}\frac{|\partial_{E}K|}{|K|}>0,\label{am}
\end{eqnarray}
where $\partial_{E} K$ consists of all the edges in $E(G)$ that have exactly one endpoint in $K$ and one endpoint not in $K$. If the left-hand side of (\ref{am}) is equal to 0, then the graph $G$ is called amenable.

 A manifold $M$ is plane if every self-avoiding cycle splits it into two parts. Examples of plane includes the 2D sphere $\bS^2$, the Euclidean plane $\RR^2$ and the hyperbolic plane $\HH^2$. See \cite{CFKP} for background on hyperbolic geometry. We say the graph $G$ is planar if it can be embedded in the plane, i.e., it can be drawn on the plane in such a way that its edges intersect only at their endpoints. We say that an embedded graph $G\subset M$ in $M$ is properly embedded if every compact subset of $M$ contains finitely many vertices of $G$ and intersects finitely many edges.
 
 The number of ends of a connected graph is the supremum over its finite subgraphs of the number of
infinite components that remain after removing the subgraph. The number of ends and planarity of a graph is closely related to properties of statistical mechanical models on the graph; see \cite{ZLSAW}, for example, about the effects of the number of ends of a graph on the speed of self-avoiding walks; see \cite{GrL}, about the effects of planarity of a graph on the number of self-avoiding walks.
 
 \subsection{Ising Model and Random-Cluster Model}
 
 The random cluster measure $RC:=RC_{p,q}^{G_0}$ on a finite graph $G_0=(V_0,E_0)$ with parameters $p\in[0,1]$ and $q\geq 1$ is the probability measure on $\{0,1\}^{E_0}$ which assigns probability
\begin{eqnarray}
RC(\xi):\propto q^{k(\xi)}\prod_{e\in E_0}p^{\xi(e)}(1-p)^{1-\xi(e)}.\label{drc}
\end{eqnarray}
 to each $\xi\in \{0,1\}^{E_0}$, where $k(\xi)$ is the number of connected components in $\xi$.

Let $G=(V,E)$ be an infinite, connected, locally finite graph. For each $q\in[1,\infty)$ and each $p\in (0,1)$, let $WRC_{p,q}^G$ be the random cluster measure with the wired boundary condition, and let $FRC_{p,q}^G$ be the random cluster measure with the free boundary condition. More precisely, $WRC_{p,q}^G$  (resp.\ $FRC_{p,q}^G$) is the weak limit of $RC$'s defined by (\ref{drc}) on larger and larger finite subgraphs approximating $G$, where all the edges outside each finite subgraph are required to have state 1 (resp.\ state 0).

When there is no confusion, we may write $FRC_{p,q}^G$ and $WRC_{p,q}^G$ as $FRC_{p,q}$ and $WRC_{p,q}$ for simplicity.
Assume that $G$ is transitive. Then measures $FRC_{p,q}$ and $WRC_{p,q}$ are $\mathrm{Aut}(G)$-invariant, and $\mathrm{Aut}(G)$-ergodic; see Page 295 of \cite{SR01} for explanations.

If we further assume that $G$ is unimodular, nonamenable and planar,
it is known that there exists $p_{c,q}^{w},\ p_{c,q}^{f},\ p_{u,q}^{w},\ p_{c,q}^{f}\in[0,1]$, such that $FRC_{p,q}$-a.s.~the number of infinite clusters equals
\begin{eqnarray}
\left\{\begin{array}{cc}0&\mathrm{if}\ p\leq p_{c,q}^f\\ \infty&\mathrm{if}\ p\in (p_{c,q}^f,p_{u,q}^f)\\1&\mathrm{if}\ p> p_{u,q}^f; \end{array}\right.\label{frc}
\end{eqnarray}
and $WRC_{p,q}$-a.s.~the number of infinite clusters equals
\begin{eqnarray}
\left\{\begin{array}{cc}0&\mathrm{if}\ p< p_{c,q}^w\\ \infty&\mathrm{if}\ p\in (p_{c,q}^w,p_{u,q}^w)\\1&\mathrm{if}\ p\geq p_{u,q}^w.\end{array}\right.\label{wrc}
\end{eqnarray}
see expressions (17),(18), Theorem 3.1 and Corollary 3.7 of \cite{HJL02}. 

The Ising model and the random cluster model when $q=2$ can be coupled in the following way.

\begin{lemma}\label{l213}Fix an infinite locally finite graph $G=(V(G),E(G))$, and $p\in[0,1]$. Consider the Ising model with sample space $\{\pm 1\}^{V(G)}$, and coupling constant $J\geq 0$ on each edge. Let $\mu^{f}$ (resp. $\mu^+$, $\mu^-$) be the infinite volume Gibbs measure of the Ising model with free boundary conditions (resp. ``$+$''-boundary conditions, ``$-$''-boundary conditions). Let  $\omega\in \{0,1\}^{E(G)}$ be a random edge configuration on $G$. For  each connected component $C$ of $\omega$,  pick a spin uniformly from $\{\pm 1\}$,  and assign this spin  to  all  vertices  of $C$.   Do  this  independently  for  different  connected  components of $\omega$. Then we obtain a $\{\pm 1\}^{V(G)}$-valued random spin configuration $\sigma$.
\begin{enumerate}[label=(\Alph*)]
\item If  $\omega$ is distributed according to $FRC_{p,2}$, then $\sigma$ distributed according to the Gibbs measure $\mu^f$ for the Ising model on $G$ with coupling constant $J:=-\frac{1}{2}\log(1-p)$.
\item If  $\omega$ is distributed according to  $WRC_{p,2}$, then $\sigma$ distributed  according to the Gibbs measure $\frac{\mu^++\mu^-}{2}$for the Ising model on $G$ with coupling constant $J:=-\frac{1}{2}\log(1-p)$.
\end{enumerate}
\end{lemma}

\begin{proof}See Propositions 2.2 and 2.3 of \cite{HJL02}.
\end{proof}

\subsection{XOR Ising Model}

An XOR Ising model on $G$ is a probability measure on $\sigma_{XOR}\in \{\pm 1\}^{V(G)}$, such that 
\begin{eqnarray*}
\sigma_{XOR}(v)=\sigma_1(v)\sigma_2(v),\qquad\forall v\in V(G),
\end{eqnarray*}
where $\sigma_1$, $\sigma_2$ are two i.i.d.\ Ising models with spins located on $V(G)$.

The XOR Ising model was first introduced in \cite{DB11} whose contours are conjectured to have conformally invariant distribution
 at criticality. Contours in the XOR Ising model on the 2D Euclidean square grid are later proved to have the same distribution as that of level lines of the dimer model on the square hexagon lattice ([4,8,8] lattice) embedded into the Euclidean plane $\RR^2$, see \cite{BT12}. The percolation properties for the XOR Ising model on the square grid, hexagonal lattice and triangular lattice in $\RR^2$ were studied in \cite{HL,ZL17}. The percolation properties for the XOR Ising model on the non-amenable triangular lattice in the hyperbolic plane $\HH^2$ were studied in \cite{HL,ZL17}. In this paper, we shall study the percolation properties for the XOR Ising model on non-amenable, vertex-transitive, planar graphs with one end.

\subsection{Main Results}

The main goal of this paper is to study percolation properties of Ising models in the hyperbolic plane. More precisely, we consider when there exists an infinite ``$+$''-cluster or an infinite ``$-$''-cluster for a random Ising configuration on a graph embedded into the hyperbolic plane.

Ising models in the hyperbolic plane were first studied in \cite{SS90}, where for large values of the inverse temperature an uncountable number of mutually singular Gibbs states were constructed, in contrast to the Ising model in the Euclidean plane for which any Gibbs measure is a convex combination of two extremal ones (see \cite{Ai80}) and there is a unique phase transition whose critical temperature can be explicitly identified (see \cite{Li12,CD13}). A related result was proved later in \cite{Wu1, Wu2}, more precisely, under certain temperature, the Gibbs measure with free boundary condition is not the average of the Gibbs measure with ``$+$''-boundary condition and the Gibbs measure with ``$-$''-boundary condition for the Ising model on regular tilings of the hyperbolic plane. A major tool to study the Ising model is the random cluster model (see \cite{FK72,SW87,EDS88,ACCN88}). The random cluster model in the hyperbolic plane was studied in \cite{HJL02} with applications to the Ising and Potts models. In this paper, we shall compare the percolation properties of the Ising model in the hyperbolic plane with the percolation properties of its random cluster representation, and prove the surprising result that percolation does not always occur simultaneously in the two coupled models. Similar result was proved when the Ising model is defined on a vertex-transitive triangular tiling of the hyperbolic plane; see \cite{HL,ZL17}. By applying the new techniques recently developed in \cite{ZL20}, we now extend the result to Ising models on all the connected, locally finite, vertex-transitive, non-amenable and planar graphs with one end.

We use $G^+$ to denote the dual graph of a planar graph $G$. More precisely, each vertex of $G^+$ corresponds to a face of $G$; two vertices of $G^+$ are joined by an edge in $G^+$ if and only if their corresponding faces in $G$ share an edge.
To each site configuration $\omega\in \{0,1\}^{V(G)}$, we associate a bond configuration $\phi^+_{\omega}\in \{0,1\}^{E(G^+)}$ such that for each dual edge $e^+\in E(G^+)$, $\phi^+_{\omega}(e^{+})=1$ if and only if the edge $e\in E(G)$ (dual edge of $e^+$) joins two endpoints with different states in $\omega$. 

\begin{definition}Let $G=(V,E)$ be a graph. Given a set $A\in 2^V$, and a vertex $v\in V$, denote $\Pi_{v}A=A\cup\{v\}$. For $\sA\subset 2^V$, we write $\Pi_v\sA=\{\Pi_v A: A\in \mathcal{A}\}$. A site percolation process $(\mathbf{P},\omega)$ on $G$ is insertion-tolerant if $\mathbf{P}(\Pi_v\sA)>0$ for every $v\in V$ and every event $\sA\subset 2^V$ satisfying $\mathbf{P}(\sA)>0$.

 A site percolation is deletion-tolerant if $\mathbf{P}[\Pi_{\neg v}\sA]>0$ whenever $v\in V$ and $\mathbf{P}(\sA)>0$, where $\Pi_{\neg v}A=A\setminus\{v\}$ for $A\in 2^{V}$, and $\Pi_{\neg v}\sA=\{\Pi_{\neg v} A: A\in \mathcal{A}\}$.
\end{definition}

It is straightforward to check that the Ising percolation with any coupling constant $J\in \RR$ is both deletion tolerant and insertion tolerant.

.
\begin{theorem}\label{m1}Let $G$ be a connected, vertex-transitive, locally finite, planar, nonamenable graph with one end. Consider an $\mathrm{Aut}(G)$-invariant, $\mathrm{Aut}(G)$-ergodic, insertion-tolerant and deletion-tolerant site percolation $\omega\in\{0,1\}^{V(G)}$ with distribution $\mu$. Let $s_0$ (resp.\ $s_1, k^+$) be the total number of infinite 0-clusters in $\omega$(resp.\ infinite 1-clusters in $\omega$, infinite contours in $\phi^+_{\omega}$), then
\begin{eqnarray*}
(s_0,s_1,k^+)\in\{(0,0,1),(\infty,0,\infty),(0,\infty,\infty),(0,1,0),(1,0,0),(\infty,\infty,\infty)\}\ a.s.
\end{eqnarray*}

\end{theorem}

Applying Theorem \ref{m1} on the Ising percolation, we obtain the following results. 

\begin{theorem}\label{ipl}Let $G$ be a connected, vertex-transitive, locally finite, planar, nonamenable graph with one end. Let $d\geq 3$ be the vertex degree of $G$. Assume that the critical site percolation probability on $G$ satisfies
\begin{eqnarray}
p_c^{site}(G)<\frac{1}{2}.\label{psh}
\end{eqnarray}
Consider the Ising model with spins located on vertices of $G$ and coupling constant $J\in \geq$ on each edge. Let $\omega\in \{\pm 1\}^{V(G)}$ be an Ising configuration.
\begin{enumerate}
\item Let $h>0$ satisfy
\begin{eqnarray}
\frac{e^{-h}}{e^{h}+e^{-h}}=p_c^{site}(G)\label{pch}
\end{eqnarray}
  Let $\mu^+$ (resp.\ $\mu^-$. $\mu^f$) be the infinite-volume Ising Gibbs measure with ``$+$''-boundary conditions (resp.\ ``$-$'' boundary conditions, free boundary conditions).  If 
 \begin{eqnarray}
 0\leq J<\frac{h}{d},\label{jh} 
 \end{eqnarray}
 then $\mu$-a.s. there are infinitely many infinite ``$+$''-clusters and infinitely many infinite ``$-$''-clusters in $\omega$, and infinitely many infinite contours in $\phi^+_{\omega}$, where $\mu$ is an arbitrary $\mathrm{Aut}(G)$-invariant Gibbs measure for the Ising model on $G$ with coupling constant $J$.
   \item  Assume $J\geq 0$. If one of the following conditions 
 \begin{enumerate}[label=(\alph*)]
\item  $\mu^f$ is $\mathrm{Aut}(G)$-ergodic; 
\item $\mathrm{inf}_{u,v\in V(G)}\langle \sigma_{u}\sigma_{v}\rangle_{\mu^f}=0$, where $\sigma_{u}$ and $\sigma_{v}$ are two spins associated to vertices $u,v\in V(G)$ in the Ising model; 
\item $0\leq J<\frac{1}{2}\ln\left(\frac{1}{1-p_{u,2}^{f}}\right)$, where $p_{u,2}^{f}$ is the critical probability for the existence of a unique infinite open cluster of the corresponding random cluster representation of the Ising model on $G$, with free boundary conditions as given in (\ref{frc});  
\end{enumerate} 
  holds, then $\mu^f$-a.s. there are infinitely many infinite ``$+$''-clusters and infinitely many infinite ``$-$''-clusters. Indeed, we have $(c)\Rightarrow(b)\Rightarrow (a)$.
  \end{enumerate}

\end{theorem}

Theorem \ref{ipl} discusses conditions on the coupling constant $J$ such that the Ising configuration on a non-amenable, transitive, planar, one-ended graph has infinitely many infinite ``$+$''-clusters and infinitely many infinite ``$-$''-clusters. In Theorem \ref{tm3}, we shall compare the Ising percolation with the percolation in its random-cluster representation, and prove that percolation does not alway occur simultaneously in an Ising model and its random cluster representation.

\begin{theorem}\label{tm3}Let $G$ be a connected, vertex-transitive, locally finite, planar, nonamenable, one-ended graph with vertex degree $d$. Then 
\begin{enumerate}
\item If $p_c^{site}(G)>\frac{1}{2}$, then 
\begin{eqnarray}
p_{c,2}^{w}(G)\geq 1-\left(\frac{1-p_c^{site}}{p_c^{site}}\right)^{\frac{1}{d}}>0.\label{pcwl}
\end{eqnarray}
\item If $p_c^{site}(G)<\frac{1}{2}$, and
\begin{eqnarray}
J <\frac{1}{2}\log\left(\frac{1}{1-p^w_{c,2}}\right);\label{jj3}
\end{eqnarray}
then for any Gibbs measure of the Ising model on $\omega\in\{\pm 1\}^{V(G)}$ with coupling constant $J$, a.s.~there are infinitely many infinite ``$+$"-clusters and infinitely many infinite ``$-$"-clusters in $\omega$, and infinitely many infinite contours in $\phi^+_{\omega}$. However, for any Gibbs measure of the random cluster representation of the Ising model  on $\{0,1\}^{E(G)}$, a.s.~there are no infinite 1-clusters.
\end{enumerate}
\end{theorem}

\noindent{\textbf{Remark.}} If $G$ is a connected, transitive, planar, one-ended graph with vertex degree $d\geq 7$, then $G$ must be non-amenable, and $p_c^{site}(G)<\frac{1}{2}$; see Lemma \ref{lc1}.

The next theorem discusses percolation properties in the XOR Ising model.

\begin{theorem}\label{t16}Let $G$ be a connected, vertex-transitive, locally finite, planar, nonamenable graph with one end. Let $d\geq 3$ be the vertex degree of $G$. Assume that the critical site percolation probability on $G$ satisfies (\ref{psh}).
Consider the XOR Ising model with spins located on vertices of $G$ such that $\sigma_{XOR}=\sigma_1\sigma_2$; where $\sigma_1,\sigma_2\in\{\pm\}^{V(G)}$ are two i.i.d.~Ising configurations with coupling constant $J\geq 0$ on each edge.

 Let $h>0$ satisfy
\begin{eqnarray}
\frac{2}{(e^{h}+e^{-h})^2}=p_c^{site}(G).\label{pch1}
\end{eqnarray}
If 
 \begin{eqnarray}
 0\leq J<\frac{h}{d},\label{jh1} 
 \end{eqnarray}
 then $\mu\times \mu$-a.s. there are infinitely many infinite ``$+$''-clusters and infinitely many infinite ``$-$''-clusters in $\sigma_{XOR}$, and infinitely many infinite contours in $\phi^+_{\sigma_{XOR}}$, where $\mu$ is an arbitrary $\mathrm{Aut}(G)$-invariant Gibbs measure for the Ising model on $G$ with coupling constant $J$.
 \end{theorem}
 
 Using planar duality of the XOR Ising model, we obtain the following theorem.
 
\begin{theorem}\label{xorc} Let $\sigma_1$, $\sigma_2$ be two i.i.d.\ Ising models with spins located on vertices of the graph $G$, and coupling constant $K\geq 0$. Let $\sigma_{XOR}=\sigma_1\sigma_2$. Let $J\geq 0$ be given by
\begin{eqnarray}
e^{-2J}=\frac{1-e^{-2K}}{1+e^{-2K}},\label{jkr}
\end{eqnarray}
and let $k^+$ be the number of infinite contours in $\phi^+_{\sigma_{XOR}}$. Assume $J$ satisfies (\ref{pch1}) and (\ref{jh1}).  Let $\mu^+$ (resp.\ $\mu^-$. $\mu^f$) be the infinite-volume Ising Gibbs measure with ``$+$''-boundary conditions (resp.\ ``$-$'' boundary conditions, free boundary conditions) and coupling constant $K$ on each edge of $G$, then
\begin{eqnarray*}
\mu^+\times \mu^+(k^+\in\{0,\infty\})=\mu^-\times \mu^-(k^+\in\{0,\infty\})=\mu^f\times \mu^f(k^+\in\{0,\infty\})=1.
\end{eqnarray*}
\end{theorem}

The organization of the paper is as follows. In Section \ref{bk}, we review some known results about percolation and Ising model in the hyperbolic plane, which will be used to prove main results in this paper. In Section \ref{p11}, we prove Theorem \ref{m1}. In Section \ref{p123}, we prove Theorems \ref{ipl} and \ref{tm3}. In Section \ref{p5}, we prove Theorems \ref{t16} and \ref{xorc}.

\section{Backgrounds}\label{bk}

 In this section, we review some known results about percolation and Ising model in the hyperbolic plane, which will be used to prove main results in this paper.

An Archimedean tiling of a two-dimensional Riemannian manifold is a tiling by regular polygons such that the group of isometries of the tiling acts transitively on the vertices of the tiling. 

\begin{lemma}\label{p21}Let $G$ be a locally finite, connected, vertex-transitive planar graph with at most one end. The $G$ has an embedding on $\mathbb{S}^2$, $\RR^2$ or $\HH^2$ as an 
Archimedean tiling; all automorphisms of $G$ extend to automorphisms of the tiling and are induced by isometries of the geometry.
\end{lemma}

\begin{proof}See Theorem 3.1 of \cite{Bab97}.
\end{proof}

For an vertex-transitive Archimedean tiling, there is an simple criterion to determine whether the graph is amenable or not (see \cite{R04}).

\begin{lemma}\label{l22}Assume the graph $G$ can be realized as a vertex-transitive Archimedean tiling on $\mathbb{S}^2$, $\RR^2$ or $\HH^2$. Assume that each vertex had degree $d\geq 3$, and is incident to $d$ faces of degree $m_1,m_2,\ldots,m_d$. 
\begin{enumerate}
\item If $\frac{1}{m_1}+\frac{1}{m_2}+\ldots+\frac{1}{m_d}=\frac{d-2}{2}$, then $G$ is infinite and amenable can be embedded into the Euclidean $\mathbb{R}^2$ such that all automorphisms of $G$ extend to automorphisms of the tiling and are induced by isometries of $\RR^2$;
\item If $\frac{1}{m_1}+\frac{1}{m_2}+\ldots+\frac{1}{m_d}>\frac{d-2}{2}$, then $G$ is finite and can be embedded into the sphere $\mathbb{S}^2$ such that all automorphisms of $G$ extend to automorphisms of the tiling and are induced by isometries of $\mathbb{S}^2$.;
\item If $\frac{1}{m_1}+\frac{1}{m_2}+\ldots+\frac{1}{m_d}<\frac{d-2}{2}$, then $G$ is non-amemable can be embedded into the hyperbolic plane $\mathbb{H}^2$ such that all automorphisms of $G$ extend to automorphisms of the tiling and are induced by isometries of $\HH^2$.
\end{enumerate}
\end{lemma}

\begin{lemma}\label{p25}Let $G=(V(G),E(G))$ be a connected, locally finite, quasi-transitive graph. Consider an invariant percolation on $G$. Assume one of the following two conditions holds
\begin{enumerate}
\item the percolation on $G$ is insertion tolerant; or 
\item $G$ is a non-amenable, planar graph with one end;
\end{enumerate}
then the number of infinite 1-clusters is a.s.~$0,1,\infty$.
\end{lemma}

\begin{proof}For the conclusion under condition (1), see Lemma 2.9 of \cite{ZL20}; see also \cite{ns81}. For the conclusion under condition (2), see Lemma 3.5 of \cite{bs00}.
\end{proof}



\begin{lemma}\label{l06}Let $G$ be a connected, non-amenable, quasi-transitive, unimodular graph, and let $\omega$ be an invariant percolation on $G$ which has a single component a.s. Then $p_c(\omega)<1$ a.s.
\end{lemma}

\begin{proof}See Theorem 3.4 of \cite{bs00}.
\end{proof}

\begin{lemma}\label{n1f}Let $G=(V(G),E(G))$ be a connected, non-amenable, locally finite, planar, transitive graph with one end. Let $\mu$ be an automorphism-invariant percolation measure on $\{0,1\}^{V(G)}$. Then
\begin{enumerate}
\item If $\mu$ is insertion-tolerant and $\mu$-a.s.~there is a unique infinite 0-cluster, then $\mu$-a.s. there are no infinite 1-clusters.
\item If $\mu$ is deletion-tolerant, and $\mu$-a.s.~there is a unique infinite 1-cluster, then $\mu$-a.s. there are no infinite 0-clusters.
\end{enumerate}
\end{lemma}

\begin{proof}See Theorem 1.6 of \cite{ZL20}.
\end{proof}

\begin{lemma}\label{lc1}Let $G$ be an infinite, connected, locally finite, transitive, planar graph in which each vertex has degree at least 7. Consider the Bernoulli($p$) site percolation of $G$. Then 
\begin{enumerate}[label=(\Alph*)]
\item $p_c^{site}<\frac{1}{2}$.
\item For every $p$ in  the range $(p_c^{site}, 1-p_c^{site})$,  there  are  infinitely  many  infinite  open clusters and infinitely many infinite closed clusters a.s.   
\item For every $p$ in the range $[0,1]$, a.s.~there exists at least 1 infinite open or closed cluster.
\end{enumerate}
\end{lemma}

\begin{proof}See Theorem 1.7 of \cite{ZL20}.
\end{proof}



Let $G$ be a connected, locally finite, transitive, non-amenable planar graph with one end. By Lemmas \ref{p21} and \ref{l22},  
we can identify the graph $G$ with its embedding in $\HH^2$ in which the action of $\Gamma$ on $G$ extends to an isometric action on $\HH^2$. Recall that $G^{+}$ is the planar dual graph of $G$.

We shall always use $*^+$ to denote the dual of $*$. If $*$ is an edge, then $*^+$ is its dual edge. If $*$ is a vertex, then $*^+$ is its dual face. If $*$ is a face, then $*^+$ is its dual vertex.

We may also consider a bond configuration $\psi\in\{0,1\}^{E(G)}$, that is, to each edge $e\in E(G)$, $\psi$ assigns a unique state in $\{0,1\}$. A contour in $\psi$ is a maximal connected set of edges of $G$ in which each edge has state 1 in $\psi$. A contour is finite (resp.\ infinite) if it contains finitely many edges (resp.\ infinitely many edges).  Each bond configuration $\psi\in\{0,1\}^{E(G)}$ also induces a bond configuration $\psi^{+}\in \{0,1\}^{E(G^+)}$ by the following rule
\begin{itemize}
\item for each $e^+\in E(G^+)$, $\psi^+(e^+)=1$ if and only if $\psi(e)=0$.
\end{itemize}

\begin{lemma}\label{l27}Let $G$ be an infinite, connected, locally finite, planar, transitive, nonamenable graph with one end. Let $\psi\in\{0,1\}^{E(G)}$ be an automorphism-invariant random bond configuration on $G$. Let $k$ be the number of infinite contours in $\psi$, and $k^+$ be the number of infinite contours in $\psi^+$. Then a.s.
\begin{eqnarray*}
(k,k^+)\in \left\{(1,0),(0,1),(1,\infty),(\infty,1),(\infty,\infty)\right\}
\end{eqnarray*}
\end{lemma}

\begin{proof}See Theorem 3.1 of \cite{bs96}.
\end{proof}

\begin{definition}(Stochastic Domination) Let $G=(V(G),E(G))$ be a graph. Let $\Omega=\{0,1\}^{E(G)}$ (resp.\ $\Omega=\{0,1\}^{V(G)}$). Then the configuration space $\Omega$ is a partially ordered set with partial order given by $\omega_1\leq \omega_2$ if $\omega_1(e)\leq \omega_2(e)$ for all $e\in E(G)$ (resp.\ $\omega_1(v)\leq \omega_2(v)$ for all $v\in V(G)$). A random variable $X:\Omega\rightarrow\RR$ is  called increasing if $X(\omega_1)\leq X(\omega_2)$ whenever $\omega_1\leq \omega_2$. An  event $A\subset\Omega$ is called increasing (respectively, decreasing) if its indicator function $1_A$ is increasing  (respectively, decreasing). Given two probability measures $\mu_1$, $\mu_2$ on $\Omega$, we write $\mu_1\prec \mu_2$, and we say that $\mu_2$ \textbf{stochastically dominates} $\mu_1$, if $\mu_1(A)\leq \mu_2(A)$ for all increasing events $A\subset \Omega$.
\end{definition}

\begin{lemma}(Holley inequality)\label{hln} Let $G=(V(G),E(G))$ be a finite graph. Let $\Omega=\{0,1\}^{V(G)}$. Let $\mu_1$ and $\mu_2$ be strictly positive probability measures on $\Omega$ such that
\begin{eqnarray}
\mu_2(\max\{\omega_1,\omega_2\})\mu_1(\min(\omega_1,\omega_2))\geq \mu_1(\omega_1)\mu_2(\omega_2),\qquad \omega_1,\omega_2\in\Omega,\label{hli}
\end{eqnarray}
Then 
\begin{eqnarray*}
\mu_1\prec \mu_2.
\end{eqnarray*}
\end{lemma}
\begin{proof}See Theorem (2.1) of \cite{GrGrc}; see also \cite{HR74}.
\end{proof}


\begin{lemma}\label{liff}Let $G$ be a graph, $J>0$, and $p=1-e^{-2J}$. There is no infinite cluster $WRC_{p,2}$ a.s.~if and only if there is a unique Gibbs measure for the Ising model with coupling constant $J$.
\end{lemma}

\begin{proof}See Proposition 3.2 (i) of \cite{HJL02}.
\end{proof}

\begin{definition}\label{idd}Let $G=(V,E)$ be a graph and $\Gamma$ a transitive group acting on $G$. Suppose that $X$ is either $V$, $E$ or $V\cup E$. Let $Q$ be a measurable space and $\Omega:=2^V\times Q^X$. A probability measure $\mathbf{P}$ on $\Omega$ will be called a \textbf{site percolation with scenery} on $G$. The projection onto $2^V$ is the underlying percolation and the projection onto $Q^X$ is the scenery. If $(\omega,q)\in \Omega$, we set $\Pi_v(\omega,q)=(\Pi_v \omega,q)$. We say the percolation with scenery $\mathbf{P}$ is \textbf{insertion-tolerant} if $\mathbf{P}(\Pi_v\mathcal{B})>0$ for every measurable $\mathcal{B}\subset \Omega$ with positive measure. We say that $\mathbf{P}$ has \textbf{indistinguishable infinite clusters} if for every $\mathcal{A}\subset 2^V\times 2^V\times Q^X$ that is invariant under diagonal actions of $\Gamma$, for $\mathbf{P}$-a.e. $(\omega,q)$, either all infinite clusters $C$ of $\omega$ satisfy $(C,\omega,q)\in \mathcal{A}$, or they all satisfy $(C,\omega,q)\notin \mathcal{A}$.
 \end{definition}
 
 \begin{proposition}\label{idc}Let $\mathbf{P}$ be a site percolation with scenery on a graph $G=(V,E)$ with state space $\Omega:=2^V\times Q^X$, where $Q$ is a measurable space and $X$ is either $V$, $E$ or $V\cup E$. If $\mathbf{P}$ is $\Gamma$-invariant and insertion-tolerant, then $\mathbf{P}$ has indistinguishable infinite clusters.
 \end{proposition}
 
 \begin{proof}See Theorem 3.3, Remark 3.4 of \cite{LS99}.
 \end{proof}

\section{Numbers of Infinite Clusters and Infinite Contours}\label{p11}

This section is devoted to the proof of Theorem \ref{m1} about the possible numbers of infinite clusters and infinite contours in a percolation configuration. The idea is to list all the possible values of $(s_0,s_2,k^{+})$, and then exclude those that a.s.~cannot occur by planarity, ergodicity and symmetry.

By Lemma \ref{p25}, a.s.
\begin{eqnarray*}
(s_0,s_1)\in\{(0,0),(0,1),(0,\infty),(1,0),(1,1),(1,\infty),(\infty,0),(\infty,1),(\infty,\infty)\}.
\end{eqnarray*}
Since the probability measure is both insertion-tolerant and deletion-tolerant, by Lemma \ref{n1f} a.s.
\begin{eqnarray*}
(s_0,s_1)\in\{(0,0),(0,1),(0,\infty),(1,0),(\infty,0),(\infty,\infty)\}.
\end{eqnarray*}
We shall investigate the number $k^+$ of infinite contours in $\phi^+_{\omega}$ for each pair of possible values of $(s_0,s_1)$. Without loss of generality, assume that $\mu$ is ergodic. Note that each configuration $\omega$ induces a bond configuration $\phi_{\omega}\in \{0,1\}^{E(G)}$, such that for each $e=(u,v)\in E(G)$, $\phi_{\omega}(e)=1$ if and only if $\omega(u)=\omega(v)$. Let $k$ be the total number of infinite contours in $\phi_{\omega}$. Then $k=s_0+s_1$.

Let $\overline{G}$ be the superposition of $G$ and $G^{+}$. More precisely, each vertex of $\overline{G}$ is either a vertex of $G$, a vertex of $G^{+}$ or the midpoint of an edge of $G$. Two vertices $u,v$ of $\overline{G}$ are joined by an edge of $\overline{G}$ if and only if one of the following two conditions holds.
\begin{enumerate}
\item $u$ is a vertex of $G$, and $v$ is the midpoint of an edge $e(v)$ of $G$, such that $e(v)$ is incident to $u$, or vice versa;
\item $u$ is a vertex of $G^{+}$, and $v$ is the midpoint of an edge $e^+(v)$ of $G^+$, such that $e^+(v)$ is incident to $u$, or vice versa.
\end{enumerate}

Let $\overline{G}^{+}$ be the dual graph of $\overline{G}$. Since in $\overline{G}$ each face has degree 4, in $\overline{G}^+$ each vertex has degree 4.  See Figure \ref{fig:spd}.

\begin{figure}
\includegraphics[width=.4\textwidth]{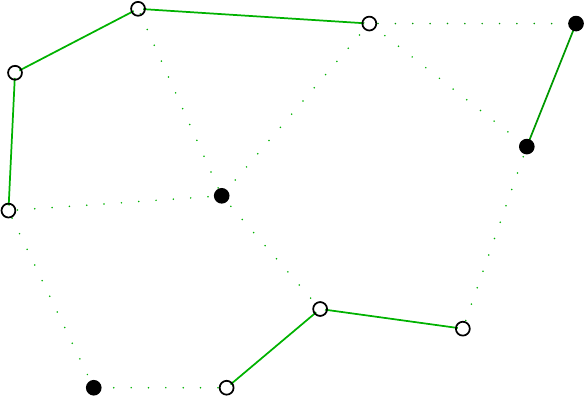}\qquad\qquad \includegraphics[width=.4\textwidth]{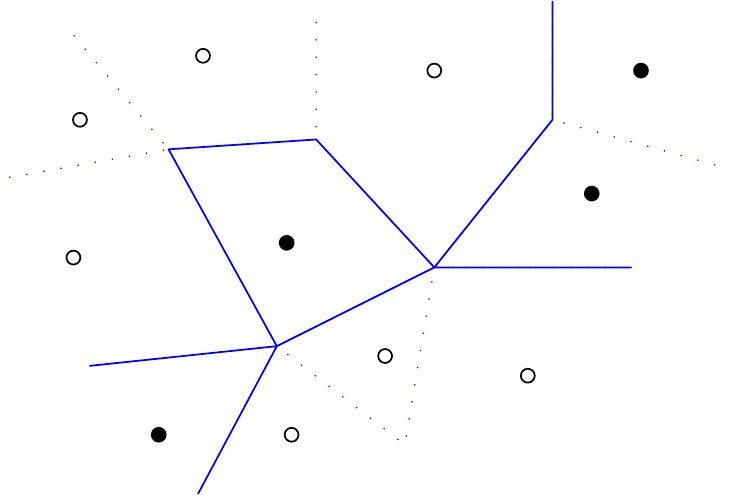}\\
\bigskip
\includegraphics[width=.4\textwidth]{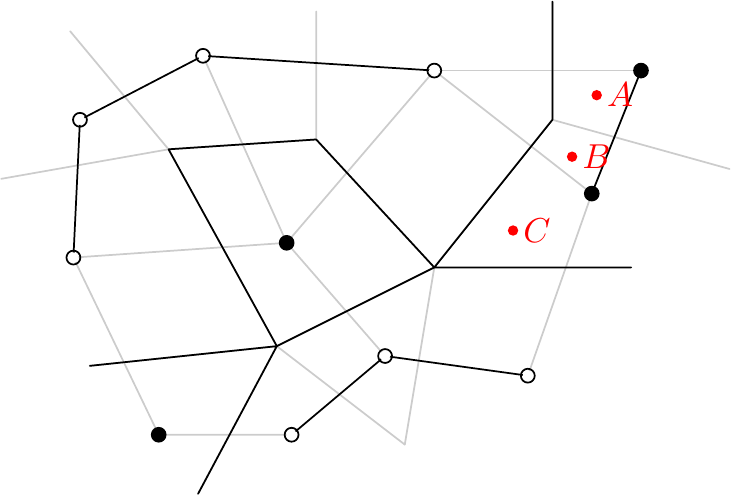}\qquad\qquad \includegraphics[width=.4\textwidth]{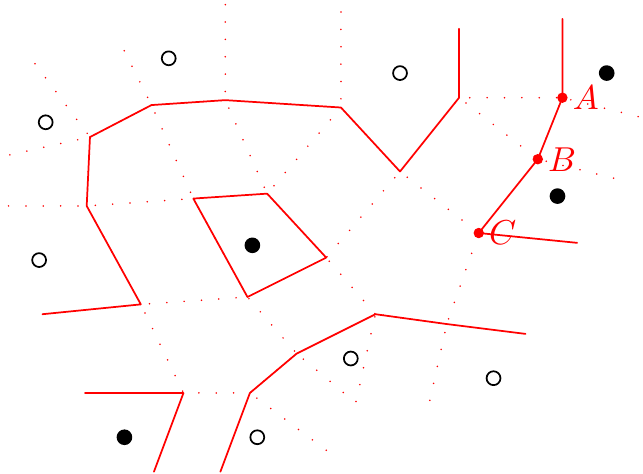}
\caption{State-1 vertices in $\omega$ are represented by dots; state-0 in $\omega$ are represented by circles. The upper left figure represents the graph $G$; state-1 edges in $\phi_{\omega}$ are represented by solid green lines; state-0 edges in $\phi_{\omega}$ are represented by dotted green lines.  The upper right figure represents the graph $G^+$; state-1 edges in $\phi^+_{\omega}$ are represented by solid blue lines; state-0 edges in $\phi^+_{\omega}$ are represented by dotted blue lines. The lower left figure represents the graph $\ol{G}$; state-1 edges in $\ol{\phi}_{\omega}$ are represented by black lines; state-0 edges in $\ol{\phi}_{\omega}$ are represented by grey lines The lower right figure represents the graph $\ol{G}^+$; state-1 edges in $\eta_{\omega}$ are represented by solid red lines; state-0 edges in $\eta_{\omega}$ are represented by dotted red lines.}
\label{fig:spd}
\end{figure}

Note that each vertex of $G^{+}$ has an even degree in the subgraph $\phi^+_{\omega}$. That is because an edge in $E(G^+)$ is present (has state 1) in $\phi^+_{\omega}$ if and only if it separates two vertices of $G$ with different state in $\omega$. Winding once around each face of $G$, the states of vertices must change an even number of times to make sure that each vertex has a unique state in $\omega$.

The configurations $\phi_{\omega}$ and $\phi^+_{\omega}$ naturally induce a bond configuration $\overline{\phi}_{\omega}\in \{0,1\}^{E(\overline{G})}$, where for each $f\in E(\overline{G})$, $\overline{\phi}_{\omega}(f)=1$ if and only if $f$ is the half edge of an edge $e\in E(G)\cup E(G^+)$ satisfying either $\phi_{\omega}(e)=1$ or $\phi^+_{\omega}(e)=1$.

 We define the interface $\eta_{\omega}$ for $\omega$ to be a bond configuration in $\{0,1\}^{E(\overline{G}^{+})}$, where an edge $f\in E(\overline{G}^+)$ satisfies $\eta_{\omega}(f)=1$ if and only if $\overline{\phi}_{\omega}(f^+)=0$. A contour in $\phi_{\omega}$ (resp.\ $\phi^{+}_{\omega}$, $\ol{\phi}_{\omega}$, $\eta_{\omega}$) is a maximal connected set of edges in $E(G)$ (resp.\ $E(G^+)$, $E(\ol{G})$, $E(\ol{G}^+)$) such that each edge has state 1 in $\phi_{\omega}$ (resp.\ $\phi^{+}_{\omega}$, $\ol{\phi}_{\omega}$, $\eta_{\omega}$).
 
Throughout this section we shall use the following notation:
\begin{itemize}
\item $s_0$: the total number of infinite 0-clusters in $\omega$;
\item $s_1$: the total number of infinite 1-clusters in $\omega$;
\item $k$: the total number of infinite contours in $\phi_{\omega}$; note that $k=s_0+s_1$;
\item $k^{+}$: the total number of infinite contours in $\phi^+_{\omega}$;
\item $t$: the total number of infinite contours in $\ol{\phi}_{\omega}$; note that $t=k+k^+$;
\item $t^+$: the total number of infinite contours in $\eta_{\omega}$.
\end{itemize}

We say two infinite clusters $A$, $B$ in $\omega$ are adjacent if there exists a path $l_{ab}$, joining a vertex $a\in A$ and $b\in B$, and consisting of edges of $G$, such that $l_{ab}$ does not intersect any other infinite clusters in $\omega$. In particular, if there are exactly two infinite clusters in $\omega$, then the two infinite clusters must be adjacent.
 Let $L$ be an infinite contour in $\phi^{+}_{\omega}$. We say $A$ is incident to $L$ if there exists a vertex $z\in A$ and an edge $e^{+}$ in $L$ such that $z$ is an endpoint of the dual edge $e$ of $e^{+}$.

\begin{lemma}\label{l41}Each contour in $\eta_{\omega}$ is either a self-avoiding cycle or a doubly infinite self-avoiding path.
\end{lemma}

\begin{proof}See Lemma 3.1 of \cite{ZL20}.
\end{proof}

\begin{lemma}\label{l42}Let $\omega\in \{0,1\}^{V(G)}$. Assume that $I$ is a doubly infinite self-avoiding path consisting of edges of $E(\ol{G}^+)$, which is also an infinite contour in $\eta_{\omega}$. Then $I$ splits the hyperbolic plane $\HH^2$ into two unbounded components. Exactly one component (denoted by $P_I$) contains an infinite contour in $\phi_{\omega}$, and the other component (denoted by $D_I$) contains an infinite contour in $\phi^+_{\omega}$. 
\begin{enumerate}
\item Let $V_I\subset V(G)$ consist of all the vertices on all the faces of $G$ crossed by I. Then all the vertices in $V_I\cap P_I$ are in the same infinite cluster of $\omega$.
\item Let $V_I^+\subset V(G^+)$ consist of all the vertices on all the faces of $G^+$ crossed by I. Then all the vertices in $V_I^+\cap D_I$ are in the same infinite contour of $\phi^+_{\omega}$.
\item If the total number of infinite 0-clusters and infinite 1-clusters in $\omega$ is 1. Denote the unique infinite cluster in $\omega$ by $\xi$, then $\xi\subset P_I$.
\item If there exists a unique infinite contour $C$ in $\phi^+_{\omega}$, then $C\subset D_I$
\end{enumerate}
\end{lemma}

\begin{proof}It is straightforward to check the lemma from the construction of $\eta_{\omega}$.
\end{proof}

\begin{lemma}\label{l43}Let $G$ be a graph satisfying the condition of Theorem \ref{m1}. Let $\omega\in \{0,1\}^{V(G)}$ be an invariant percolation on $G$ with distribution $\mu$.  Then a.s. $t^+\neq 1$. 
\end{lemma}

\begin{proof}Without loss of generality, assume that $\mu$ is ergodic. If $t^+=1$ a.s., then the unique infinite contour in $\eta_{\omega}$ forms an invariant bond percolation on $\overline{G}^+$ which has a single component a.s. It is straightforward to check that $\overline{G}^+$ is non-amenable, quasi-transitive and unimodular (quasi-transitive planar graphs are unimodular, see \cite{LP}). This contradicts Lemma \ref{l06} since the unique infinite contour in $\eta_{\omega}$ is a doubly infinite self-avoiding path by Lemma \ref{l41}, which has critical percolation probability 1. 
\end{proof}

\begin{lemma}\label{l44}Let $G$ be a graph satisfying the condition of Theorem \ref{m1}. Let $\omega\in \{0,1\}^{V(G)}$ be an invariant percolation on $G$ with distribution $\mu$. If $\mu$ is both insertion-tolerant and deletion-tolerant, and $\mu$-a.s.~$s_0+s_1=\infty$, then $\mu$-a.s.~$k^+\neq 1$.
\end{lemma}

\begin{proof}

 Assume that $k^+=1$ a.s., we shall obtain a contradiction.  In this case 
 \begin{eqnarray*}
 t=s_0+s_1+k^{+}=\infty,\ a.s. 
\end{eqnarray*} 
 By Lemma \ref{l27}, a.s.~$t^+\in\{1,\infty\}$. By Lemma \ref{l43}, a.s.~$t^+\in\{\infty\}$. Let $C$ be the unique infinite contour in $\phi^+_{\omega}$. 
Let $\mathcal{I}$ be the collection of all the infinite contours in $\eta_{\omega}$. By Lemma \ref{l42} (4), we have
\begin{eqnarray*}
C\subset \cap_{I\in \mathcal{I}}D_I.
\end{eqnarray*} 

Then we claim that for any $I,J\in \mathcal{I}$ such that $I$ and $J$ are disjoint, we have $I\subset D_J$. Indeed, if there exist two disjoint infinite contours $I,J\in\mathcal{I}$ satisfying $I\subset P_{J}$, since both $I$ and $J$ are doubly infinite self-avoiding paths, then by Lemma \ref{l42} (2) all the vertices in $V_{I}^+\cap D_I$ is in the same infinite contour $C_1$ of $\phi^+_{\omega}$. Similarly, all the vertices in $V_J^{+}\cap D_J$ are also in the same infinite contour of $\phi^+_{\omega}$ denoted by $C_2$.  Moreover, $C_1$ and $C_2$ must be two distinct infinite contours in $\phi^+_{\omega}$ because they are separated by the infinite cluster in $\omega$ including $V_J\cap P_J$. See Figure \ref{fig:curve}. But this contradicts the assumption that $k^+=1$ a.s.

\begin{figure}
\includegraphics[width=.5\textwidth]{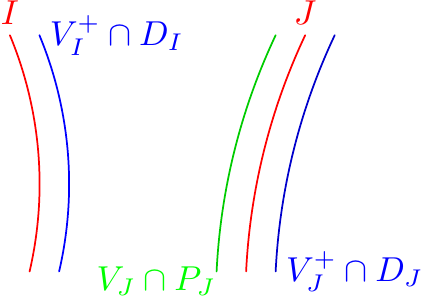}
\caption{Red lines represent two infinite contours $I$ and $J$ in $\eta_{\omega}$. Blue lines represent infinite contours in $\phi^+_{\omega}$. Green lines represent infinite contours in $\phi_{\omega}$.}\label{fig:curve}
\end{figure}

Therefore for any two disjoint $I,J\in \mathcal{I}$, $P_{I}\cap P_{J}=\emptyset$, and $V_I\cap P_I$ and $V_J\cap P_J$ are in two distinct infinite clusters of $\omega$. Hence $\{V_I\cap P_I\}_{I\in \mathcal{I}}$ are in infinitely many infinite clusters in $\omega$.

Note that one of the following two cases must occur.
\begin{enumerate}[label=(\alph*)]
\item There exist at least two distinct infinite contours $I_3, I_4\in \mathcal{I}$, such that each one of $V_{I_3}\cap P_{I_3}$ and $V_{I_4}\cap P_{I_4}$ is contained in an infinite 1-cluster of $\omega$; or
\item There exist at least two distinct infinite contours $I_3, I_4\in \mathcal{I}$, such that each one $V_{I_3}\cap P_{I_3}$ and $V_{I_4}\cap P_{I_4}$ is contained in an infinite 0-cluster of $\omega$; 
\end{enumerate}
 
 We shall prove the conclusion of the lemma when (a) occurs; the conclusion of the lemma under (b) can be proved using similar arguments.

Assume that (a) occurs. Then we can find two distinct infinite contours $I_3$ and $I_4$ in $\eta_{\omega}$ satisfying
\begin{itemize}
\item there exists two vertices $x,y\in  V(G)$, such that there are two edges $(x,c),(y,d)\in \overline{G}$, satisfying $(x,c)^+\in I_3$, $(y,d)^+\in I_4$; and
\item $x\in P_{I_3}$, $y\in P_{I_4}$ and $\omega(x_3)=\omega(x_4)=1$;
\item $C\subset D_{I_3}\cap D_{I_4}$ by Lemma \ref{l42} (4).
\end{itemize} 
Let $l_{xy}$ be a path joining $x$ and $y$ and consisting of edges of $G$. Define a new configuration $\bar{\omega}\in \{0,1\}^{V(G)}$ by 
\begin{eqnarray*}
\bar{\omega}(z)=\begin{cases}1,&\mathrm{if}\ z\in l_{uv};\\ \omega(z)&\mathrm{otherwise.}\end{cases},\ z\in V(G).
\end{eqnarray*}

\begin{figure}
\includegraphics[width=.45\textwidth]{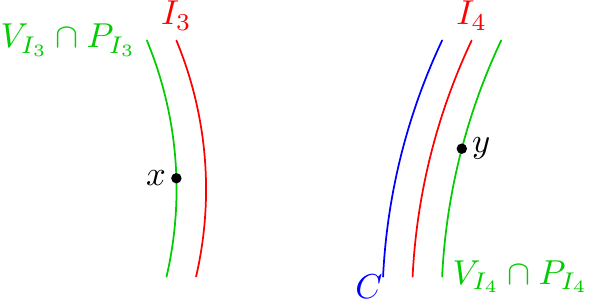}\qquad\qquad\includegraphics[width=.3\textwidth]{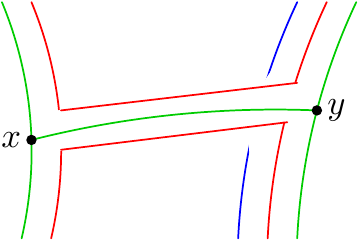}
\caption{In the left (resp.\ right) figure, Red lines represent infinite contours in $\eta_{\omega}$ (resp.\ $\eta_{\ol{\omega}}$). Blue lines represent infinite contours in $\phi^+_{\omega}$ (resp.\ $\phi^+_{\ol{\omega}}$). Green lines represent infinite contours in $\phi_{\omega}$ (resp.\ $\phi_{\ol{\omega}}$).}\label{fig:contoursplit}
\end{figure}

 As in Lemma \ref{l42}(2), the infinite contour  $C$ including $V_{I_3}^+\cap D_{I_3}$ splits into two infinite contours in $\phi^+_{\bar{\omega}}$ by the path $l_{x,y}$.  See Figure \ref{fig:contoursplit}. Therefore in $\phi^+_{\bar{\omega}}$ there exists at least 2 infinite contours. By insertion-tolerance, there exist at least 2 infinite contours in $k^+$ with positive probability, but this is a contradiction to the assumption that $k^+=1$ a.s.
The contradiction implies the lemma.
\end{proof}

\begin{lemma}\label{l45}Let $G$ be a graph satisfying the condition of Theorem \ref{m1}. Let $\omega\in \{0,1\}^{V(G)}$ be an invariant percolation on $G$ with distribution $\mu$. 
\begin{enumerate}
\item If $\mu$ is both deletion-tolerant, and $\mu$-a.s.~$(s_0,s_1)=(0,1)$, then $\mu$-a.s.~$k^+=0$.
\item If $\mu$ is both insertion-tolerant, and $\mu$-a.s.~$(s_0,s_1)=(1,0)$, then $\mu$-a.s.~$k^+=0$.
\end{enumerate}
\end{lemma}

\begin{proof}We only prove Part (1) here, Part (2) can be proved using similar arguments.

Assume $(s_0,s_1)=(0,1)$ a.s. Then $k=s_0+s_1=1$ a.s. By Lemma \ref{l27}, either $k^+=0$ a.s., or $k^+=\infty$ a.s. 

Assume that $k^+=\infty$ a.s., we shall obtain a contradiction. In this case $t=k+k^+=\infty$ a.s. By Lemma \ref{l27}, a.s.~$t^+\in\{1,\infty\}$. By Lemma \ref{l43} a.s.~$t^+=\infty$. Let $\omega\in \{0,1\}^{V(G)}$ be a configuration such that $(s_0,s_1,k^+,t^+)=(0,1,\infty,\infty)$.  Let $\xi_1$ be the unique infinite 1-cluster in $\omega$. Then we can find two distinct infinite contours $I_1$ and $I_2$ in $\eta_{\omega}$ satisfying
\begin{enumerate}[label=(\alph*)]
\item there exists two vertices $u,v\in \xi_1\cap V(G)$, such that there are two edges $(u,a),(v,b)\in \overline{G}$, satisfying $(u,a)^+\in I_1$, $(v,b)^+\in I_2$;
\item $\xi_1\subset P_{I_1}\cap P_{I_2}$ by Lemma \ref{l42} (3).
\item $u,v\in \xi_1$.
\end{enumerate} 
We claim that $D_{I_1}\cap D_{I_2}=\emptyset$. Indeed, if $D_{I_1}\cap D_{I_2}\neq\emptyset$, either $P_{I_1}\cap P_{I_2}=\emptyset$, which contradicts Part (b) above; or one of $P_{I_1}$ and $P_{I_2}$ is a subset of the other. Without loss of generality, assume that $P_{I_1}\subset P_{I_2}$. Then by Lemma \ref{l42}, $V_{I_1}^+\cap D_{I_1}$ is in an infinite contour of $\phi_{\omega}^+$ separating two infinite 1-clusters in $\omega$, one containing $V_{I_1}\cap P_{I_1}$, the other containing $V_{I_2}\cap P_{I_2}$. But this is a contradiction to the fact that there is a unique infinite 1-cluster.

Let $l_{uv}$ be a path joining $u$ and $v$ and consisting of edges of $G$. Define a new configuration $\tilde{\omega}\in \{0,1\}^{V(G)}$ by 
\begin{eqnarray*}
\tilde{\omega}(z)=\begin{cases}0,&\mathrm{if}\ z\in l_{uv};\\ \omega(z)&\mathrm{otherwise.}\end{cases},\ z\in V(G).
\end{eqnarray*}
 As in Lemma \ref{l42}(1), the infinite 1-cluster $\xi_1$ including $V_{I_1}\cap P_{I_1}$ splits into two infinite 1-clusters in $\tilde{\omega}$ by the path $l_{uv}$. Therefore in $\tilde{\omega}$ there exists at least 2 infinite 1-clusters. See Figure \ref{fig:cfc}.
 \begin{figure}
\includegraphics[width=.32\textwidth]{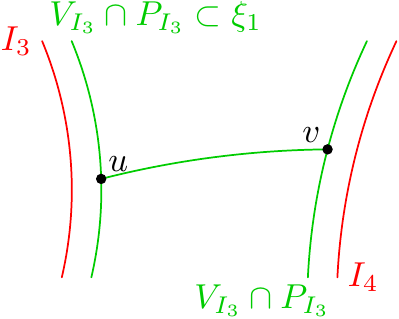}\qquad\qquad\includegraphics[width=.32\textwidth]{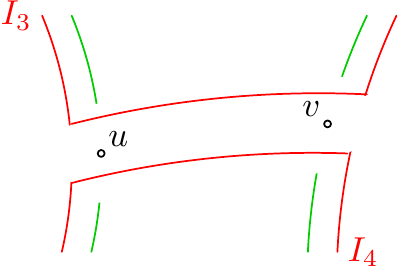}
\caption{In the left (resp.\ right) figure, Red lines represent infinite contours in $\eta_{\omega}$ (resp.\ $\eta_{\tilde{\omega}}$).  Green lines represent infinite contours in $\phi_{\omega}$ (resp.\ $\phi_{\tilde{\omega}}$).}\label{fig:cfc}
\end{figure}

  By deletion-tolerance, there exist at least 2 infinite 1-clusters with positive probability, but this is a contradiction to the assumption that $(s_0,s_1)=(0,1)$ a.s.
 Then we concluded that in this case $(s_0,s_1,k^+)=(0,1,0)$ a.s.
\end{proof}

\subsection{$(s_0,s_1)=(0,0)$, a.s.}\label{ss00} Then $k=0$ a.s.  By Lemma \ref{l27}, $k^+=1$ a.s., hence in this case, $(s_0,s_1,k^+)=(0,0,1)$ a.s.

\subsection{$(s_0,s_1)=(0,1)$, a.s.}\label{ss01} Then $(s_0,s_1,k^+)=(0,1,0)$ a.s.~by Lemma \ref{l45}.

 \subsection{$(s_0,s_1)=(0,\infty)$, a.s.}\label{sszi} Then $k=\infty$ a.s. By Lemma \ref{l27}, either $k^+=1$ a.s., or $k^+=\infty$ a.s. By Lemma \ref{l44} we concluded that in this case $(s_0,s_1,k^+)=(0,\infty,\infty)$ a.s. 
 
\subsection{$(s_0,s_1)=(1,0)$, a.s.} Then $(s_0,s_1,k^+)=(1,0,0)$ a.s.~by Lemma \ref{l45}.

\subsection{$(s_0,s_1)=(\infty,0)$, a.s.} Using the same arguments as in \ref{sszi} and deletion-tolerance, we obtain that in this case $(s_0,s_1,k^+)=(\infty,0,\infty)$ a.s.

\subsection{$(s_0,s_1)=(\infty,\infty)$ a.s.}\label{ssff} Then $k=\infty$ a.s. By Lemma \ref{l27}, either $k^+=1$ a.s., or $k^+=\infty$ a.s.
By Lemma \ref{l44} we concluded that in this case $(s_0,s_1,k^+)=(\infty,\infty,\infty)$ a.s. 

Then the theorem follows from Sections \ref{ss00}-\ref{ssff}.

\section{Ising Percolation }\label{p123}

In this section, we prove Theorems \ref{ipl} and \ref{tm3} about percolation properties of the Ising model on the hyperbolic plane. We shall compare the Ising measure with the Bernoulli($p$) percolation measure, and obtain stochastic domination result to study infinite ``$+$''-clusters and infinite ``$-$''-clusters in the Ising configuration.

\begin{lemma}\label{l51}Let $G$ be an infinite, connected, vertex-transitive graph with finite vertex-degree $d$. Let $0\leq p_2<p_1\leq 1$, and let $J\geq 0$. Let $\nu_1$ (resp.\ $\nu_2$) be the probability measure for the i.i.d.~Bernoulli site percolation on $G$ in which each vertex takes the value ``$+$'' with probability $p_1$ (resp.\ $p_2$) and the value ``$-$'' with probability $1-p_1$ (resp.\ $1-p_2$) satisfying 
\begin{eqnarray*}
p_2<\frac{e^{-dJ}}{e^{dJ}+e^{-dJ}}<\frac{e^{dJ}}{e^{dJ}+e^{-dJ}}<p_1
\end{eqnarray*}
Let $\mu^+$ (resp.\ $\mu^-$) be the probability measure for the Ising model on $G$ with coupling constant $J$ on each edge and ``$+$'' boundary conditions (resp.\ ``$-$'' boundary conditions). Let $\mu$ be an arbitrary $\mathrm{Aut}(G)$-invariant probability measure for the Ising model on $G$ with coupling constant $J$.  Then we have
\begin{eqnarray*}
\nu_2\prec \mu^-\prec \mu\prec \mu^+\prec \nu_1.
\end{eqnarray*}
\end{lemma}

\begin{proof}

Fix a face $F_0$ of $G$. Let $B_R=(V(B_R),E(B_R))$ be the finite subgraph of $G$ consisting of all the faces of $G$ whose graph distance to $F_0$ is at most $R$. Let $\nu_{1,R}$ (resp.\ $\nu_{2,R}$) be the restriction of $\nu_1$ (resp.\ $\nu_2$) on $B_R$. Let $\mu_{R}^{+}$ (resp. $\mu_{R}^{-})$ be the probability measure for the Ising model on $B_R$ with respect to the coupling constant $J$ and the ``+'' boundary condition (resp.\  the ``$-$'' boundary condition). Let $\omega_1$, $\omega_2$ be two configurations in $\{-1,1\}^{V(B_R)}$.  Then by Lemmas \ref{hln}, we can check the F.K.G.~lattice conditions below
\begin{eqnarray*}
\nu_{1,R}(\max\{\omega_1,\omega_2\})\mu_{R}^+(\min\{\omega_1,\omega_2\})\geq \nu_{1,R}(\omega_1)\mu_R^+(\omega_2)\\
\mu_{R}^-(\max\{\omega_1,\omega_2\})\nu_{2,R}(\min\{\omega_1,\omega_2\})\geq \mu_R^-(\omega_1)\nu_{2,R}(\omega_1).
\end{eqnarray*}
Then we obtain the following stochastic domination result:
\begin{eqnarray*}
\nu_{2,R}\prec\mu_R^{-}\prec \mu_R^+\prec\nu_{1,R}.
\end{eqnarray*}
Letting $R\rightarrow\infty$, then the theorem follows.
\end{proof}

\subsection{Proof of Theorem \ref{ipl}(1)}

First note that if (\ref{psh}) and (\ref{pch}) holds, then 
\begin{eqnarray}
\frac{e^{h}}{e^{h}+e^{-h}}< 1-p_c^{site}(G).\label{puh}
\end{eqnarray}

Assume that $J$ satisfies (\ref{jh}). Let $\nu_1$ (resp.\ $\nu_2$) be the probability measure for the i.i.d.~Bernoulli site percolation on $G$ in which each vertex takes the value ``$+$'' with probability $p_1$ (resp.\ $p_2$) satisfying 
\begin{eqnarray*}
&&\frac{e^{dJ}}{e^{dJ}+e^{-dJ}}<p_1<1-p_c^{site}(G)\\
&& p_c^{site}(G)<p_2<\frac{e^{-dJ}}{e^{dJ}+e^{-dJ}}
\end{eqnarray*}
and the value ``$-$'' with probability $1-p_1$ (resp.\ $1-p_2$). Such $p_1$ and $p_2$ exist by (\ref{jh}).

By Lemma $\ref{l51}$ we have $\nu_2\prec \mu^{-}\prec \mu\prec \mu^+\prec\nu_{1}$.

Since $p_2>p_c^{site}(G)$, $\nu_2$-a.s. there are infinite ``$+$''-clusters. By stochastic domination $\mu$-a.s. there are infinite ``$+$''-clusters. Similarly, by Lemma \ref{lc1} $\nu_1$-a.s.~there are infinite ``$-$''-clusters, and therefore $\nu$-a.s. there are infinite ``$-$''-clusters. By Theorem \ref{m1}, we conclude that when (\ref{jh}) hold, $\mu$-a.s. there are infinitely many infinite ``$+$''-clusters and infinitely many infinite ``$-$''-clusters in $\omega$, and infinitely many infinite contours in $\phi^+_{\omega}$. This completes the proof of Part (1).

\subsection{Proof of Theorems \ref{ipl}(2)}

We first prove Part (2)(a). Let $\omega\in \{\pm 1\}^{V(G)}$. Let $s_{+}$ (resp.\ $s_-$) be the number of infinite ``$+$''-clusters (resp.\ infinite ``$-$''-clusters) in $\omega$. Since the measure $\mu^{f}$ is ergodic and symmetric with respect to interchanging the ``$+$'' states and the ``$-$'' states, we obtain by Theorem \ref{m1} that $\mu$.a.s.~
\begin{eqnarray}
(s_+,s_-)\in\{(0,0),(\infty,\infty)\}\label{pmif}
\end{eqnarray}

Let $h$ be given by (\ref{pch}). If $|J|<\frac{h}{d}$ and $\mu^f$ is $\Aut(G)$-ergodic, then the conclusion follows form Theorem \ref{ipl}(1). Hence it suffices to prove the conclusion when $J\geq \frac{h}{d}$ and $\mu^f$ is $\Aut(G)$-ergodic. But when $J\geq \frac{h}{d}$, we can find $0<J'<\frac{h}{d}$, such that 
\begin{eqnarray*}
\mu_{J'}^f\prec\mu_{J}^f
\end{eqnarray*}
Since $\mu_{J'}^f$-a.s.~there exist infinite ``$+$''-clusters, $\mu_{J}^f$ there exist infinite ``$+$''-clusters as well. Then $\mu_J^f$-a.s.~there exist infinitely many infinite ``$+$''-clusters and infinitely many infinite ``$-$''-clusters by (\ref{pmif}). This completes the proof of Part (2)(a).

It remains to prove $(c)\Rightarrow (b)\Rightarrow (a)$.
The statement $(b)\Rightarrow (a)$ follows from Theorem 4.1 of \cite{SR01}.

The fact that $(c)\Rightarrow (a)$ follows from Theorem 3.2 (v) of \cite{HJL02}; while the fact that $(c)\Rightarrow (b)$ follows from Theorem 4.1 and Lemma 6.4 of \cite{LS99}.


\subsection{Proof of Theorem \ref{tm3}(1)}

Let $d\geq 3$ be the vertex degree of the graph $G$. Let $h>0$ be such that 
\begin{eqnarray}
p_1:=\frac{e^h}{e^h+e^{-h}}<p_c^{site}\label{hu}
\end{eqnarray}
and assume that the coupling constant $J\geq 0$ of the Ising model on $G$ satisfies
\begin{eqnarray}
0<J<\frac{h}{d}\label{ju}
\end{eqnarray}
Let $p_2=1-p_1> 1-p_c^{site}$. Let $\nu_1$ (resp.\ $\nu_2$) be the probability measure for the i.i.d.~Bernoulli site percolation on $G$ in which each vertex takes the value ``$+$'' with probability $p_1$ (resp.\ $p_2$), and takes the value ``$-$'' with probability $1-p_1$ (resp.\ $1-p_2$). By Lemma $\ref{l51}$ we have $\nu_2\prec \mu^{-}\prec \mu\prec \mu^+\prec\nu_{1}$. Since $p_1<p_c^{site}$, $\nu_1$-a.s.~there are no infinite ``$+$''-clusters, therefore $\mu^{+}$-a.s.~there are no infinite ``$+$''-clusters in the Ising model. Since $\mu^-\prec\mu^+$, $\mu^-$-a.s.~there are no infinite ``$+$''-clusters either. By symmetry we obtain that $\mu^{+}$-a.s.~there are neither infinite ``$+$"-clusters nor infinite ``$-$''-clusters; and $\mu^{-}$-a.s.~there are neither infinite ``$+$"-clusters nor infinite ``$-$''-clusters either.

  Let
\begin{eqnarray*}
p:=1-e^{-2J}>0;
\end{eqnarray*}
By the coupling of the Ising model and the random cluster model in Lemma \ref{l213}, each infinite 1-cluster in the random cluster representation of the Ising model must be a subset of an infinite (``$+$'' or ``$-$'') cluster in the Ising configuration. Since $\mu^{+}$-a.s.~there are neither infinite ``$+$"-clusters nor infinite ``$-$''-clusters; and $\mu^{-}$-a.s.~there are neither infinite ``$+$"-clusters nor infinite ``$-$''-clusters, by Lemma \ref{l213}(B), $WRC_{p,2}$-a.s.~there are no infinite 1-clusters in the random cluster representation of the Ising model. Since $p>0$, we obtain that $p_{c,2}^w\geq p>0$. Taking supreme over all the $(j,h)$'s satisfying (\ref{hu}), (\ref{ju}), we obtain (\ref{pcwl}).

\subsection{Proof of Theorem \ref{tm3}(2)}

 By Lemma \ref{liff}, if (\ref{jj3}) holds, then there is a unique infinite-volume Gibbs measure for the Ising model on $G$ with coupling constant $J$. Since
\begin{eqnarray*}
p_{c,2}^{w}\leq p_{c,2}^{f}\leq p_{u,2}^{f},
\end{eqnarray*}
(\ref{jj3}) implies Condition (c). Then Theorem \ref{ipl} (2)(c) implies $\mu^f$ a.s.~there are infinitely many infinite ``$+$''-clusters and infinitely many ``$-$''-clusters.
By the uniqueness of the infinite-volume Gibbs measure, there are infinitely many infinite ``$+$''-clusters and infinitely many ``$-$''-clusters under any infinite-volume Gibbs measure for the Ising model. But there are no infinite 1-clusters in the random cluster representation of the Ising model by (\ref{wrc}). 

\section{XOR Ising Percolation}\label{p5}

In this section, we prove Theorems \ref{t16} and \ref{xorc} about percolation properties of the XOR Ising model in the hyperbolic plane. In the proof of Theorem \ref{t16},  we shall again compare the XOR Ising measure with the Bernoulli($p$) percolation measure, and obtain stochastic domination result to study infinite ``$+$''-clusters and infinite ``$-$''-clusters in the XOR Ising configuration. In the proof of Theorem \ref{xorc}, we shall apply the planar duality of the XOR Ising model, which was first proved for the XOR Ising model on the Euclidean square grid in \cite{BT12}.

\begin{lemma}\label{l61}Let $G$ be an infinite, connected, vertex-transitive graph with finite vertex-degree $d$. Let $0\leq p_2<p_1\leq 1$, and let $J\geq 0$. Let $\nu_1$ (resp.\ $\nu_2$) be the probability measure for the i.i.d.~Bernoulli site percolation on $G$ in which each vertex takes the value ``$+$'' with probability $p_1$ (resp.\ $p_2$) and the value ``$-$'' with probability $1-p_1$ (resp.\ $1-p_2$) satisfying 
\begin{eqnarray}
p_2<\frac{2}{(e^{dJ}+e^{-dJ})^2}<\frac{e^{2dJ}+e^{-2dJ}}{(e^{dJ}+e^{-dJ})^2}<p_1\label{xoc}
\end{eqnarray}
 Let $\mu$ be an arbitrary $\mathrm{Aut}(G)$-invariant probability measure for the Ising model on $G$ with coupling constant $J$.  Then we have
\begin{eqnarray*}
\nu_2\prec[ \mu\times \mu]\prec \nu_1.
\end{eqnarray*}
\end{lemma}

\begin{proof}

Fix a face $F_0$ of $G$. Let $B_R=(V(B_R),E(B_R))$ be the finite subgraph of $G$ consisting of all the faces of $G$ whose graph distance to $F_0$ is at most $R$. Let $\nu_{1,R}$ (resp.\ $\nu_{2,R}$) be the restriction of $\nu_1$ (resp.\ $\nu_2$) on $B_R$.  Let $\mu_R$ be a probability measure for an Ising model on $B_R$ with coupling constant $J$ and certain boundary conditions such that $\lim_{R\rightarrow\infty}\mu_R=\mu$. Let $\sigma_{XOR}=\sigma_1\sigma_2$. Then for each $v\in V(B_R)$,
\begin{eqnarray*}
&&\mu_{R}\times \mu_{R}(\sigma_{XOR}(v)=``+")\\
&=&\mu_{R}(\sigma_1=``+")\mu_{R}(\sigma_2=``+")+\mu_{R}(\sigma_1=``-")\mu_{R}(\sigma_2=``-")\\
&&\mu_{R}\times \mu_{R}(\sigma_{XOR}(v)=``-")\\
&=&\mu_{R}(\sigma_1=``+")\mu_{R}(\sigma_2=``-")+\mu_{R}(\sigma_1=``-")\mu_{R}(\sigma_2=``+")
\end{eqnarray*}
Hence we have
\begin{eqnarray*}
\frac{2}{(e^{dJ}+e^{-dJ})^2}\leq \mu_{R}\times \mu_{R}(\sigma_{XOR}(v)=``+")\leq \frac{e^{2dJ}+e^{-2dJ}}{(e^{dJ}+e^{-dJ})^2}\\
\frac{2}{(e^{dJ}+e^{-dJ})^2}\leq \mu_{R}\times \mu_{R}(\sigma_{XOR}(v)=``-")\leq \frac{e^{2dJ}+e^{-2dJ}}{(e^{dJ}+e^{-dJ})^2}
\end{eqnarray*}

 Let $\omega_1$, $\omega_2$ be two configurations in $\{-1,1\}^{V(B_R)}$.  By (\ref{xoc}), we can check the F.K.G.~lattice conditions below
\begin{eqnarray*}
&&\nu_{1,R}(\max\{\omega_1,\omega_2\})[\mu_R\times\mu_R](\min\{\omega_1,\omega_2\})\geq \nu_{1,R}(\omega_1)[\mu_{R}\times\mu_R](\omega_2)\\
&&[\mu_R\times\mu_R](\max\{\omega_1,\omega_2\})\nu_{2,R}(\min\{\omega_1,\omega_2\})\geq[\mu_R\times\mu_R](\omega_1)\nu_{2,R}(\omega_1).
\end{eqnarray*}
Then by Lemmas \ref{hln} we obtain the following stochastic domination result:
\begin{eqnarray*}
\nu_{2,R}\prec[\mu_R\times\mu_R]\prec\nu_{1,R}.
\end{eqnarray*}
Letting $R\rightarrow\infty$, then the theorem follows.
\end{proof}

\subsection{Proof of Theorem \ref{t16}}

First note that if (\ref{psh}) and (\ref{pch1}) holds, then 
\begin{eqnarray}
\frac{e^{2h}+e^{-2h}}{(e^{h}+e^{-h})^2}=1-p_c^{site}(G).\label{puh1}
\end{eqnarray}

Assume that $J$ satisfies (\ref{jh}). Let $\nu_1$ (resp.\ $\nu_2$) be the probability measure for the i.i.d.~Bernoulli site percolation on $G$ in which each vertex takes the value ``$+$'' with probability $p_1$ (resp.\ $p_2$) satisfying 
\begin{eqnarray*}
&&\frac{e^{2dJ}+e^{-2dJ}}{(e^{dJ}+e^{-dJ})^2}<p_1<1-p_c^{site}(G)\\
&& p_c^{site}(G)<p_2<\frac{2}{(e^{dJ}+e^{-dJ})^2}
\end{eqnarray*}
and the value ``$-$'' with probability $1-p_1$ (resp.\ $1-p_2$). Such $p_1$ and $p_2$ exist by (\ref{jh}).

By Lemma $\ref{l51}$ we have $\nu_2\prec [\mu\times\mu]\prec\nu_{1}$.

Since $p_2>p_c^{site}(G)$, $\nu_2$-a.s. there are infinite ``$+$''-clusters. By stochastic domination $[\mu\times\mu]$-a.s. there are infinite ``$+$''-clusters in the XOR Ising configuration. Similarly, by Lemma \ref{lc1} $\nu_1$-a.s.~there are infinite ``$-$''-clusters, and therefore $\nu$-a.s. there are infinite ``$-$''-clusters. By Theorem \ref{m1}, we conclude that when (\ref{jh}) hold, $\mu$-a.s. there are infinitely many infinite ``$+$''-clusters and infinitely many infinite ``$-$''-clusters in $\sigma_{XOR}$, and infinitely many infinite contours in $\phi^+_{\sigma_{XOR}}$. This completes the proof of the Theorem.  $\hfill\Box$

\subsection{Proof of Theorem \ref{xorc}} 

Let $\Lambda=(V_{\Lambda},E_{\Lambda})$ be a subgraph of $G^+$ consisting of faces of $G^+$. Let $\Lambda_{*}=(V_{\Lambda_*},E_{\Lambda_*})$ be the dual graph of $\Lambda$, such that there is a vertex in $V_{\Lambda}$ corresponding to each face in $\Lambda$, as well as the unbounded face; the edges in $E_{\Lambda}$ and $E_{\Lambda_*}$ are in 1-1 correspondence by duality.
 
 Consider an XOR Ising model on $\Lambda$ with respect to two i.i.d.~Ising models $\sigma_3$, $\sigma_4$ with free boundary conditions and coupling constants $J\geq 0$ satisfying (\ref{pch1}) and (\ref{jh1}). The partition function of the XOR Ising model can be computed by
 \begin{eqnarray*}
 Z_{\Lambda,f}=\sum_{\sigma_3,\sigma_4\in \{\pm 1\}^{V_{\Lambda}}}\prod_{(u,v)\in E_{\Lambda}}e^{J(\sigma_{3,u}\sigma_{3,v}+\sigma_{4,u}\sigma_{4,v})}.
 \end{eqnarray*}
 
 Following the same computations as in \cite{BT12}, we obtain
 \begin{eqnarray}
 Z_{\Lambda,f}= C_1\sum_{P_*\in \mathcal{P}_{*},P\in \mathcal{P},P\cap P_*=\emptyset}\left(\frac{2 e^{-2J}}{1+e^{-4J}}\right)^{|P_*|}\left(\frac{1-e^{-4J}}{1+e^{-4J}}\right)^{|P|}.\label{zlf}
 \end{eqnarray}
 where $\mathcal{P}_*$ (resp.\ $\mathcal{P}$) consists of all the contour configurations on $E_{\Lambda_*}$ (resp.\ $E_{\Lambda}$) such that each vertex of $V_{\Lambda_*}$  (resp.\ $V_{\Lambda}$) has an even number of incident present edges, and $C_1=2^{|V_{\Lambda}|-|E_{\Lambda}|+2}(e^{2J}-e^{-2J})^{|E_{\Lambda}|}$ is a constant. 
 
 When $J,K$ satisfies (\ref{jkr}), we have
 \begin{eqnarray*}
 \frac{2e^{-2J}}{1+e^{-4J}}&=&\frac{1-e^{-4K}}{1+e^{-4K}};\\
 \frac{2e^{-2K}}{1+e^{-4K}}&=&\frac{1-e^{-4J}}{1+e^{-4J}}.
 \end{eqnarray*}
Thus the partition function $Z_{\Lambda,f}$, up to a multiplicative constant, is the same as the partition function of the XOR Ising model on $\Lambda_*$ with coupling constant $K$. 

Recall that there is exactly one vertex $v_{\infty}\in V_{\Lambda_{*}}$ corresponding to the unbounded face in $\Lambda$. The XOR Ising model $\sigma_{XOR}=\sigma_1\sigma_2$ on $\Lambda_*$, corresponds to an XOR Ising model on $\Lambda_{*}\setminus \{v_{\infty}\}$ (which is a subgraph of $G$) with the boundary condition that all the boundary vertices have the same state in $\sigma_1$ and all the boundary vertices have the same state in $\sigma_2$. Since $\sigma_1$ and $\sigma_2$ are i.i.d., the boundary condition must be one of the following two cases:
\begin{enumerate}
\item ``$+$'' boundary condition in $\sigma_1$ and ``$+$'' boundary condition in $\sigma_2$, denoted by ``$++$'' boundary condition for the XOR Ising model;
\item ``$-$'' boundary condition in $\sigma_1$ and ``$-$'' boundary condition in $\sigma_2$, denoted by ``$--$'' boundary condition for the XOR Ising model.
\end{enumerate}

 Note that each one of the 2 possible boundary conditions gives the same distribution of contours in the XOR Ising model. From the expression (\ref{zlf}), we can see that there is a natural probability measure on the set of contours $\Phi=\{(P,P_*):P\in\mathcal{P},P_*\in \mathcal{P}_*,P\cap P_*=\emptyset\}$, such that the probability of each pair of contours $(P,P_*)\in \Phi$ is proportional to $\left(\frac{2 e^{-2J}}{1+e^{-4J}}\right)^{|P_*|}\left(\frac{1-e^{-4J}}{1+e^{-4J}}\right)^{|P|}$, and the marginal distribution on $\mathcal{P}$ is the distribution of contours for the XOR Ising model on $G$ with coupling constant $K$ and free boundary conditions, while the marginal distribution on $\mathcal{P}_*$ is the distribution of contours for the XOR Ising model on $G^+$ with coupling constant $J$ and free boundary conditions. 

We let $\Lambda$ and $\Lambda_*\setminus \{v_{\infty}\}$ increase and approximate the graph $G^+$ and $G$, respectively. If with a positive $\mu^+\times \mu^+$ probability, there exists exactly one infinite contour $C$ consisting of edges of $G^+$ for the XOR Ising model on $G$ with coupling constant $K$, then $\mu^f\times \mu^f$-a.s. there exists an infinite cluster consisting of vertices of  $G^+$ containing all the vertices in $C$, since contours in $G$ and $G^+$ are disjoint. Consider the XOR Ising spin configuration $\sigma_{XOR}=\sigma_1\sigma_2$ as a site percolation on $G^+$, with scenery given by contour configurations in $\{0,1\}^{E(G^+)}$ within the ``$+$'' clusters of $\sigma_{XOR}$. In the notation of Definition \ref{idd}, $Q=\{0,1\}$, and $X=E(G^+)$. An edge in $E(G^+)$ is present (has state ``1'') if and only if both of its endpoints are in a ``$+$''-cluster of the XOR Ising configuration on $G^+$ and the edge itself present in the contour configuration of the XOR Ising model on $G$. This way we obtain an automorphism-invariant and insertion-tolerant percolation with scenery. Let $\mathcal{A}\subset 2^{V(G^+)}\times 2^{V(G^+)}\times 2^{E(G^+)}$ be the triple $(C,\omega, q)$ such that 
\begin{itemize}
\item $\omega$ is an XOR Ising spin configuration on $G^+$; and
\item $C$ is an infinite ``$+$''-cluster in $\omega$; and
\item $q$ is the $G^+$-contour configuration for the XOR Ising model on $G$ such that each contour is within  a ``$+$''-cluster of $\omega$; and
\item $C$ contains an infinite contour in $q$.
\end{itemize}
We can see that $\mathcal{A}$ is invariant under diagonal actions of automorphisms. 
By Theorem \ref{t16}, $\mu^f\times\mu^f$-a.s. there exists infinitely many infinite ``$+$''-clusters in $\omega$. By Proposition \ref{idc}, either all the infinite clusters are in $\mathcal{A}$, or no infinite clusters are in $\mathcal{A}$. Similar arguments applies for ``$-$''-clusters in $\omega$. Hence almost surely the number of infinite contours in $G^+$ is 0 or $\infty$. Since the distribution of infinite contours in $G^+$ is exactly that of contours for the XOR Ising model on $G$ with coupling constant $K$ and $++$ (or $--$) boundary condition, we obtain that
\begin{eqnarray*}
\mu^+\times \mu^+(k^+\in\{0,\infty\})=\mu^-\times \mu^-(k^+\in\{0,\infty\})=1.
\end{eqnarray*}
The identity $\mu^f\times \mu^f(k^+\in\{0,\infty\})=1$ can be proved in a similar way.
  $\hfill\Box$




\bigskip
\noindent\textbf{Acknowledgements.} Z.L.'s research is supported by National Science Foundation grant 1608896 and Simons Collaboration Grant 638143.

\bibliography{perc}

\providecommand{\bysame}{\leavevmode\hbox to3em{\hrulefill}\thinspace}
\providecommand{\MR}{\relax\ifhmode\unskip\space\fi MR }
\providecommand{\MRhref}[2]{%
  \href{http://www.ams.org/mathscinet-getitem?mr=#1}{#2}
}
\providecommand{\href}[2]{#2}
\begin{thebibliography}{10}

\bibitem{Ai80}
M.~Aizenman, \emph{Translation invariance and instability of phase coexistence
  in the two dimensional {I}sing system}, Communications in Mathematical
  Physics \textbf{73} (1980), 83--94.

\bibitem{ACCN88}
M.~Aizenman, J.T. Chayes, L.~Chayes, and C.M. Newman, \emph{Discontinuity of
  the magnetization in one-dimensional $1/|x-y|^2$ {I}sing and {P}otts models},
  J. Statist. Phys. \textbf{50} (1988), 1--40.

\bibitem{Bab97}
L.~Babai, \emph{The growth rate of vertex-transitive planar graphs.},
  Proceedings of the Eighth Annual ACM-SIAM Symposium on Discrete Algorithms
  (New Orleans, LA, 1997), New York, 1997, pp.~564--573.

\bibitem{bs96}
I.~Benjamini and O.~Schramm, \emph{Percolation beyond $\mathbb{Z}^d$, many
  questions and a few answers}, Electronic Communications in Probability
  \textbf{1} (1996), 71--82.

\bibitem{bs00}
\bysame, \emph{Percolation in the hyperbolic plane}, Journal of the American
  Mathematical Society \textbf{14} (2000), 487--507.

\bibitem{BT12}
C.~Boutillier and B.~de~Tili\`ere, \emph{Height representation of xor-ising
  loops via bipartite dimers}, Electronic Journal of Probability \textbf{19}
  (2014), 33p.

\bibitem{CFKP}
J.~W. Cannon, W.~J. Floyd, R.~Kenyon, and W.~R. Parry, \emph{Hyperbolic
  geometry}, Flavors of Geometry, Cambridge Univ. Press, Cambridge, 1997,
  pp.~59--115.

\bibitem{CD13}
D.~Cimasoni and H.~Duminil-Copin, \emph{The critical temperature for the ising
  model on planar doubly periodic graphs}, Electron. J. Probab. \textbf{18}
  (2013), 18pp.

\bibitem{EDS88}
R.G. Edwards and A.D. Sokal, \emph{Generalization of the
  {F}ortuin-{K}asteleyn-{S}wendsen-{W}ang representation and monte carlo
  algorithm}, Phys. Rev. D \textbf{38} (1988), 2009--2012.

\bibitem{FK72}
C.M. Fortuin and P.W. Kasteleyn, \emph{On the random-cluster model. i.
  introduction and relation to other models}, Physica \textbf{57} (1972),
  536--564.

\bibitem{GP}
G.~Grimmett, \emph{Percolation}, Springer, 1999.

\bibitem{GrGrc}
\bysame, \emph{The random-cluster model}, Springer, 2006.

\bibitem{GrL}
G.~R. Grimmett and Z.~Li, \emph{Cubic graphs and the golden mean}, Discrete
  Mathematics \textbf{343} (2020), 11638.

\bibitem{HJL02}
O.~H\"aggstr\"om, J.~Jonasson, and R.~Lyons, \emph{Explicit isoperimetric
  constants and phase transitions in the random-cluster model}, Ann. Probab.
  \textbf{30} (2002), 443--473.

\bibitem{HR74}
R.~Holley, \emph{Remarks on the {FKG} inequalities}, Commun. Math. Phys.
  \textbf{36} (1974), 227--231.

\bibitem{HL}
A.~Holroyd and Z.~Li, \emph{Constrained percolation in two dimensions}, Annales
  de L’institut Henri Poincar\'e D (2020).

\bibitem{ZL20}
Z.~Li, \emph{Site percolation on planar graphs},
  \url{https://arxiv.org/abs/2005.04529}.

\bibitem{Li12}
\bysame, \emph{Critical temperature of periodic {I}sing models}, Communications
  in Mathematical Physics \textbf{315} (2012), 337--381.

\bibitem{ZL17}
\bysame, \emph{Constrained percolation, {I}sing model and {XOR} {I}sing model
  on planar lattices}, Random Structures and Algorithms (2020).

\bibitem{ZLSAW}
\bysame, \emph{Positive speed self-avoiding walks on graphs with more than one
  end}, Journal of Combinatorial Theory, Series A. \textbf{175} (2020), 105257.

\bibitem{LP}
R.~Lyons and Y.~Peres, \emph{Probability on trees and networks}, Cambridge
  University Press, 2016.

\bibitem{LS99}
R.~Lyons and O.~Schramm, \emph{Indistinguishability of percolation clusters},
  Ann. Probab. \textbf{27} (1999), 1809--1836.

\bibitem{ns81}
C.M. Newman and L.S. Shulman, \emph{Infinite clusters in percolation models},
  J. of Statis. Phys. \textbf{26} (1981), 613--628.

\bibitem{R04}
D.~Renault, \emph{The vertex-transitive {TLF}-planar graphs}, Discrete
  Mathematics \textbf{309} (2009), 2815--2833.

\bibitem{SR01}
R.~H. Schonmann, \emph{Multiplicity of phase transitions and mean-field
  criticality on highly non-amenable graphs}, Commun. Math. Phys. \textbf{219}
  (2001), 271--322.

\bibitem{SS90}
C.M. Series and Ya.~G. Sinai, \emph{Ising models on the lobachevsky plane},
  Communications in Mathematical Physics \textbf{128} (1990), 63--76.

\bibitem{SW87}
R.H. Swendsen and J.S. Wang, \emph{Nonuniversal critical dynamics in {M}onte
  {C}arlo simulations}, Phys. Rev. Lett. \textbf{58} (1987), 86--88.

\bibitem{DB11}
D.B. Wilson, \emph{{XOR} {I}sing loops and {G}aussian free field},
  \url{https://arxiv.org/abs/1102.3782}.

\bibitem{Wu1}
C.~Wu, \emph{Ising models on hyperbolic graphs}, J. Stats. Phys. \textbf{85}
  (1996), 251--259.

\bibitem{Wu2}
\bysame, \emph{Ising models on hyperbolic graphs {II}}, J. Stats. Phys.
  \textbf{100} (2000), 893--904.

\end{thebibliography}
\bibliographystyle{amsplain}
\end{document}